\documentclass[mathpazo]{cicp}

\usepackage{amsmath,amssymb,amsfonts,mathtools,bm,amsthm}
\usepackage{graphicx,booktabs,array,multirow}
\usepackage{algorithm,algorithmic}
\usepackage{xcolor}
\usepackage{placeins}
\usepackage{hyperref}

\graphicspath{{Figures/}}

\begin{document}

\title{History-Conditioned Inference of Memory \\ and Stochastic Forcing in Generalized Langevin Equations under L\'evy Noise}

\author[Y.~Yang et.~al.]{Yang Yang\affil{1},
      Ting Gao\affil{1}\comma\corrauth,
      Xiaoli Chen\affil{2}\comma\footnotemark[1]}
\address{\affilnum{1}\ School of Mathematics and Statistics,
          Huazhong University of Science and Technology,
          Wuhan 430074, Hubei, P.R. China.\\
          \affilnum{2}\ School of Mathematics and Physics,
          China University of Geosciences (Wuhan),
          Wuhan 430074, Hubei, P.R. China}
\emails{{\tt zachyang@hust.edu.cn} (Y.~Yang),
        {\tt tgao0716@hust.edu.cn} (T.~Gao),
        {\tt xlchen@cug.edu.cn} (X.~Chen)}

\begin{abstract}
We consider the identification of memory and stochastic forcing in a one-dimensional generalized Langevin equation with state-dependent Brownian noise and optional additive symmetric $\alpha$-stable jumps. A finite-dimensional Markov embedding provides Kramers--Moyal conditional moment relations that separate memory drift, Brownian diffusion, and jump contributions.  From increment ensembles conditioned on the observed velocity history, truncated statistics estimate the stable parameters, corrected second moments determine the Brownian amplitude, and first moments lead to a Volterra equation for the memory kernel, solved by RKHS regularization. We establish conditional finite-sample bounds for the moment estimators and a conditional $L^2$ error bound for kernel recovery with a data-selected regularization parameter. Numerical experiments examine Brownian and Brownian--L\'evy models under different forces, memory kernels, observation lengths, measurement noise, and approximate history conditioning.
\end{abstract}

\ams{45D05, 60G52, 60H10, 65R32}
\keywords{generalized Langevin equation; history-conditioned inference; memory kernel; Kramers--Moyal; inverse problems.}

\maketitle

\section{Introduction}
\label{sec:introduction}

Generalized Langevin equations (GLEs) describe the influence of unresolved degrees of freedom through history-dependent interactions. The Mori--Zwanzig formalism derives reduced equations with deterministic, memory, and fluctuating contributions by projecting out unresolved variables~\cite{zwanzig1961,mori1965,chorin2002optimal,darve2009computing,lin2021datadriven}.
Data-driven coarse-graining and the emergence of nonlocal kernel interactions in coarse-grained models have also been studied in~\cite{felsberger2019physics,du2020coarse}.
For equilibrium systems, the fluctuating force is typically colored and linked to the memory kernel through the second fluctuation--dissipation relation, which is used in GLE models for molecular dynamics, coarse-graining, and multiscale simulations~\cite{davtyan2015dynamic,lei2016data,grogan2020data,xie2024abinitio}.
However, in driven, nonequilibrium, or phenomenological GLEs, the stochastic forcing may include state-dependent or heavy-tailed fluctuations and is not necessarily determined by the memory kernel~\cite{medino2012levy,willers2024efficient}.

Recovering the memory contribution and stochastic forcing from GLE trajectories involves two coupled inverse tasks. Increments over short intervals contain effects from history-dependent drift, continuous fluctuations, and discontinuous jumps, while recovering the memory kernel requires solving an ill-conditioned first-kind Volterra problem. Errors in the estimated first conditional moment propagate to the inferred memory force and kernel. Related interactions between stochastic excitation and ill-posed reconstruction also arise in inverse problems with random sources~\cite{bao2010stochastic}.

Most existing GLE identification methods exploit equilibrium statistics or impose structural relations between the memory kernel and the stochastic forcing.
Rational approximations and Prony expansions represent memory kernels through exponential modes and permit efficient formulations with auxiliary variables~\cite{lei2016data,grogan2020data,bockius2021model,duong2022accurate}.
Correlation-based approaches adjust kernel parameters to reproduce target dynamical statistics~\cite{jung2017iterative}. Gaussian-process optimization of discretized correlation functions has also been used to extract kernels from coarsely sampled data~\cite{tepper2024accurate}.
Regularized correlation methods combine kernel representations with Sobolev or RKHS regularization to stabilize memory recovery~\cite{lang2026memory}.
Likelihood methods fit extended non-Markovian models, including models without an imposed fluctuation--dissipation relation~\cite{vroylandt2022likelihood}; Bayesian formulations estimate drift, diffusion, and memory parameters together with posterior uncertainty~\cite{willers2024efficient}.

Machine learning approaches have been used to map dynamical data to long-lived memory kernels~\cite{winter2023deep} and to construct GLE models for transient anomalous diffusion in polymers~\cite{cherchi2024mlgle}. Neural representations of memory kernels and colored-noise generators have also been learned from atomistic data under the fluctuation--dissipation constraint~\cite{xie2024abinitio}. In contrast, the direct separation of state-dependent Brownian fluctuations, stable jump contributions, and memory drift from local increment ensembles conditioned on velocity history remains less developed. 

Short-time transition statistics provide a mechanism for such a method.
Ensembles of short trajectories have also been used to approximate unknown governing equations~\cite{wu2019numerical}.
For Markov diffusion processes, the first two conditional moments characterize the local drift and diffusion coefficients through the Kramers--Moyal expansion~\cite{kramers1940brownian,moyal1949stochastic,friedrich2000extracting}.
This idea has been extended to systems driven by Brownian and L\'evy noises, where drift, diffusion, and jump measures can be inferred from transition statistics~\cite{li2021levy,li2022governing}.
Analyses based on Kramers--Moyal coefficients have also been used to derive forward Fokker--Planck descriptions for GLEs with prescribed memory structures~\cite{wang2019diffusion}.
For a GLE, conditioning only on the current velocity averages over histories with different memory contributions. Conditioning on the velocity history retains the drift information needed for kernel inference.

We use a finite-dimensional Markov embedding to establish the Kramers--Moyal conditional moment relations and then formulate the inverse problem using increment ensembles conditioned on history. The main contributions are listed below.

\begin{itemize}
\item We derive Kramers--Moyal conditional moment relations for GLEs driven by multiplicative Brownian noise, with or without additive symmetric $\alpha$-stable L\'evy forcing.  These relations separate memory drift, Brownian diffusion, and jump contributions.

\item We combine the local moment relations with truncated increment statistics and RKHS regularization adapted to the Volterra data to recover the stable parameters, Brownian amplitude, and memory kernel from increment ensembles.

\item We establish finite sample bounds for the local moment estimators and conditional error bounds for kernel recovery, and examine the method numerically across different dynamics, memory kernels, observation conditions, and errors in history matching.
\end{itemize}

The remainder of the paper is organized as follows.
Section~\ref{sec:preliminaries} introduces the GLE model, the Markov embedding, and the associated conditional moment identities.
Section~\ref{sec:method} presents the estimation framework and gives the finite sample bounds and the conditional result for kernel recovery.
Section~\ref{sec:numerical_experiments} reports numerical experiments.
Section~\ref{sec:conclusions} summarizes the results and discusses future directions.

\section{Preliminaries}
\label{sec:preliminaries}

This section introduces the GLE model and the finite-dimensional Markov representation used to formulate its conditional moments. We derive the first and second Kramers--Moyal coefficients for Brownian and Brownian--L\'evy forcing and state the Volterra relation connecting the first moment to the memory kernel.

\subsection{Model formulation}

On a finite observation interval $[0,T]$, after nondimensionalization and setting the particle mass to unity, we consider the one-dimensional generalized Langevin equation
\begin{equation}
\mathrm{d}v_t
=
\left[
F(v_t)
-
\int_0^t \gamma(t-s)v_s\,\mathrm{d}s
\right]\mathrm{d}t
+
\sigma(v_t)\,\mathrm{d}W_t
+
\kappa\,\mathrm{d}L_t^{(\alpha)},
\label{eq:gle}
\end{equation}
Here $v_t\in\mathbb R$ denotes the observed velocity, $F$ is a known deterministic force, and $\gamma\in L^1(0,T)$ is the unknown memory kernel to be recovered. The function $\sigma(v)$ is the Brownian amplitude, $\alpha\in(0,2)$ is the stability index, and $\kappa\geq0$ is the scale parameter of the additive symmetric stable component. The corresponding Brownian diffusion coefficient is $\sigma^2(v)/2$.
Here $L_t^{(\alpha)}$ is the standardized symmetric $\alpha$-stable L\'evy motion defined by $\mathbb E\!\left[e^{\mathrm i\omega L_t^{(\alpha)}}\right]=\exp\!\left(-t|\omega|^\alpha\right)$ with $0<\alpha<2$.
Under this normalization, the scaled process $\kappa L_t^{(\alpha)}$ has L\'evy measure
\begin{equation}
\nu_{\alpha,\kappa}(\mathrm{d}r)
=
c_\alpha\kappa^\alpha |r|^{-1-\alpha}\,\mathrm{d}r,
\qquad
c_\alpha
=
\frac{
\alpha 2^{\alpha-1}\Gamma\!\left((1+\alpha)/2\right)
}{
\sqrt{\pi}\,\Gamma\!\left(1-\alpha/2\right)
}.
\label{eq:symmetric_stable_measure}
\end{equation}

We regard Eq.~\eqref{eq:gle} as an effective nonequilibrium stochastic model with independent Brownian and symmetric stable drivers.
Conditioned on the velocity history up to time $t$, the memory convolution is fixed, while the future increments of $W$ and $L^{(\alpha)}$ are independent across samples in the increment ensemble. No equilibrium fluctuation--dissipation relation is imposed between the memory kernel $\gamma$ and the stochastic
parameters $\sigma$ and $(\alpha,\kappa)$. We investigate both the Brownian case $\kappa=0$ and the Brownian--L\'evy case $\kappa>0$.

\subsection{Memory representation and Markov embedding}
\label{subsec:memory_embedding}

Kramers--Moyal coefficients are local quantities defined with respect to a Markov state. The observed variable $v_t$ alone does not retain the historical information contained in Eq.~\eqref{eq:gle}.
Accordingly, we introduce auxiliary variables to represent the memory term through a finite-dimensional Markov embedding construction~\cite{lei2016data,ottobre2011asymptotic}.

Let $\psi_k\in C^1([0,T])$, $k=1,\ldots,N_{\rm emb}$, be linearly independent basis functions. The memory kernel is approximated in the finite-dimensional embedding space $\mathcal S_{\rm emb}=\operatorname{span}\{\psi_1,\ldots,\psi_{N_{\rm emb}}\}$ as
\begin{equation}
\gamma_{\rm emb}(t)
=
\sum_{k=1}^{N_{\rm emb}}u_k\psi_k(t)
=
\bm u^\top\bm\psi(t),
\qquad
\bm u=(u_1,\ldots,u_{N_{\rm emb}})^\top.
\label{eq:kernel_expansion}
\end{equation}

The basis space is chosen such that the basis functions admit a finite-dimensional realization
\begin{equation}
\bm\psi'(t)=\bm R\bm\psi(t),
\qquad
t\in[0,T],
\label{eq:basis_closure}
\end{equation}
where $\bm R\in\mathbb R^{N_{\rm emb}\times N_{\rm emb}}$ is the corresponding realization matrix. Introducing the auxiliary variables
\begin{equation}
\bm z_t
=
\int_0^t\bm\psi(t-s)v_s\,\mathrm{d}s,
\label{eq:auxiliary_variable}
\end{equation}
the closure relation in Eq.~\eqref{eq:basis_closure} gives $\mathrm{d}\bm z_t=[\bm\psi(0)v_t+\bm R\bm z_t]\mathrm{d}t$.
Therefore, when $\gamma=\gamma_{\rm emb}$, the augmented state $X_t=(v_t,\bm z_t)\in\mathbb R^{N_{\rm emb}+1}$ satisfies
\begin{equation}
\begin{aligned}
\mathrm{d}v_t
&=
\left[
F(v_t)-\bm u^\top\bm z_t
\right]\mathrm{d}t
+
\sigma(v_t)\mathrm{d}W_t
+
\kappa\,\mathrm{d}L_t^{(\alpha)},
\\
\mathrm{d}\bm z_t
&=
\left[
\bm\psi(0)v_t+\bm R\bm z_t
\right]\mathrm{d}t.
\end{aligned}
\label{eq:augmented_system}
\end{equation}

When the true kernel $\gamma$ does not belong to the embedding space
$\mathcal S_{\rm emb}$, the finite-dimensional representation introduces the residual
\begin{equation}
r_{\rm emb}(t)
=
\int_0^t
[\gamma(t-s)-\gamma_{\rm emb}(t-s)]v_s\,\mathrm{d}s.
\end{equation}
For $t\in[0,T]$, assuming
$v\in L^\infty(0,T)$ and
$\gamma-\gamma_{\rm emb}\in L^1(0,T)$, this residual satisfies
\begin{equation}
|r_{\rm emb}(t)|
\leq
\|v\|_{L^\infty(0,t)}
\|\gamma-\gamma_{\rm emb}\|_{L^1(0,t)}
\leq
\|v\|_{L^\infty(0,T)}
\|\gamma-\gamma_{\rm emb}\|_{L^1(0,T)}.
\label{eq:kernel_residual}
\end{equation}
This term characterizes the approximation error associated with the finite-dimensional embedding space $\mathcal S_{\rm emb}$.

\subsection{Kramers--Moyal coefficients}
\label{subsec:km_coefficients}

System~\eqref{eq:augmented_system} provides a Markov representation with the augmented state $X_t=(v_t,\bm z_t)$.  We formulate the following conditional moment relations in this finite-dimensional state space.

\paragraph{Brownian case.}
For $\kappa=0$, let $\Delta v_h=v_{t+h}-v_t$ and $b(v,\bm z)=F(v)-\bm u^\top\bm z$. At the conditioning state $X_t=x$, assume that, for some $h_0>0$,
\begin{equation}
\sup_{0\le s\le h_0}\mathbb E\!\left[|b(X_{t+s})|^2+|\sigma(v_{t+s})|^4\,\middle|\,X_t=x\right]<\infty.
\label{eq:brownian_moment_envelope}
\end{equation}
Integrating the Brownian equation gives
\begin{equation}
\Delta v_h=\int_t^{t+h}b(X_s)\,\mathrm{d}s
+\int_t^{t+h}\sigma(v_s)\,\mathrm{d}W_s.
\label{eq:brownian_increment_expansion}
\end{equation}
Continuity, Eq.~\eqref{eq:brownian_moment_envelope}, It\^o's isometry, and Cauchy--Schwarz give a drift second moment of order $h^2$ and a cross term of order $h^{3/2}$. Hence,
\begin{align}
\mathbb E[\Delta v_h\mid X_t=(v,\bm z)]&=[F(v)-\bm u^\top\bm z]h+o(h),\nonumber\\
\mathbb E[(\Delta v_h)^2\mid X_t=(v,\bm z)]&=\sigma^2(v)h+o(h).
\label{eq:brownian_moment_expansions}
\end{align}

Therefore, the first two Kramers--Moyal coefficients are
\begin{align}
D_B^{(1)}(v,\bm z)
&=
\lim_{h\downarrow0}
\frac{1}{h}
\mathbb E\!\left[
\Delta v_h
\,\middle|\,
X_t=(v,\bm z)
\right]
= F(v)-\bm u^\top\bm z,
\label{eq:brownian_first_limit}
\\
D_B^{(2)}(v,\bm z)
&=
\frac{1}{2}
\lim_{h\downarrow0}
\frac{1}{h}
\mathbb E\!\left[
(\Delta v_h)^2
\,\middle|\,
X_t=(v,\bm z)
\right]
= \frac{1}{2}\sigma^2(v).
\label{eq:brownian_second_limit}
\end{align}

The instantaneous velocity $v_t$ alone is generally insufficient to determine the memory contribution, since different trajectories may reach the same velocity with different histories. 
The augmented variable $\bm z_t$ incorporates this missing memory information, allowing the local drift to be expressed as $F(v)-\bm u^\top\bm z$.

\paragraph{Brownian--L\'evy case.}
For $\kappa>0$, truncated statistics provide conditional moment relations throughout the stable range $0<\alpha<2$.
We therefore characterize the jump contribution through one-dimensional jump rate identities and use truncated conditional moments for the velocity component, based on the framework of Li and Duan~\cite{li2021levy}.
Given $\varepsilon>0$ and $q>1$, the infinitesimal jump rate over the symmetric annulus $\{\varepsilon\leq|\Delta v_h|<q\varepsilon\}$ is
\begin{align}
&\lim_{h\downarrow0}
\frac{1}{h}
\mathbb P\!\left(
\varepsilon\leq|\Delta v_h|<q\varepsilon
\,\middle|\,
X_t=(v,\bm z)
\right)
\nonumber\\
&\qquad=
\int_{\varepsilon\leq |r|<q\varepsilon}
c_\alpha\kappa^\alpha |r|^{-1-\alpha}\,\mathrm{d}r
=
\frac{2c_\alpha\kappa^\alpha}{\alpha}
\varepsilon^{-\alpha}
\left(1-q^{-\alpha}\right).
\label{eq:jump_annulus_limit}
\end{align}
Because the L\'evy noise enters additively with a constant scale parameter $\kappa$, the jump rate in~\eqref{eq:jump_annulus_limit} is independent of the state $(v,\bm z)$.

To characterize the contribution of small jumps to the second conditional moment, define
\begin{equation}
J_\varepsilon(\alpha,\kappa)
=
\int_{|r|<\varepsilon}
r^2c_\alpha\kappa^\alpha |r|^{-1-\alpha}\,\mathrm{d}r
=
\frac{2c_\alpha\kappa^\alpha}{2-\alpha}
\varepsilon^{2-\alpha}.
\label{eq:symmetric_jump_correction}
\end{equation}

For the same threshold $\varepsilon$, symmetry removes the leading contribution of small jumps from the truncated first moment, while the Brownian quadratic variation and retained jumps contribute $\sigma^2(v)h$ and $J_\varepsilon(\alpha,\kappa)h$. Then, the truncated Kramers--Moyal coefficients are given by
\begin{align}
D_{BL,\varepsilon}^{(1)}(v,\bm z)
&=
\lim_{h\downarrow0}
\frac{1}{h}
\mathbb E\!\left[
\Delta v_h
\mathbf{1}_{\{|\Delta v_h|<\varepsilon\}}
\,\middle|\,
X_t=(v,\bm z)
\right]
\nonumber\\
&=
F(v)-\bm u^\top\bm z,
\label{eq:truncated_first_limit}
\\
D_{BL,\varepsilon}^{(2)}(v,\bm z)
&=
\frac{1}{2}
\lim_{h\downarrow0}
\frac{1}{h}
\mathbb E\!\left[
(\Delta v_h)^2
\mathbf{1}_{\{|\Delta v_h|<\varepsilon\}}
\,\middle|\,
X_t=(v,\bm z)
\right]
\nonumber\\
&=
\frac{1}{2}
\left[
\sigma^2(v)+J_\varepsilon(\alpha,\kappa)
\right].
\label{eq:truncated_second_limit}
\end{align}

\subsection{Volterra inverse problem}
\label{subsec:volterra_inverse}

Let $D^{(1)}$ denote $D_B^{(1)}$ in the Brownian case and $D_{BL,\varepsilon}^{(1)}$ in the Brownian--L\'evy case. In both cases, the first Kramers--Moyal relation gives
\begin{equation}
\begin{aligned}
g_{\rm emb}(t)
&=
F(v_t)-D^{(1)}(X_t)
=
\bm u^\top\bm z_t
\\
&=
(\mathcal V_v\gamma_{\rm emb})(t)
=
\int_0^t
\gamma_{\rm emb}(t-s)v_s\,\mathrm{d}s.
\end{aligned}
\label{eq:volterra_operator}
\end{equation}
The continuous Volterra operator is compact, and its discretizations are ill-conditioned. Consequently, direct inversion can amplify perturbations in the estimated memory convolution. Considerable effort has been devoted to developing regularization techniques for first-kind Volterra equations~\cite{wang2020tikhonov,lamm2000,lu2022dartr}.

\section{Proposed method}
\label{sec:method}

This section constructs the inverse procedure from
increment ensembles conditioned on history.  We first estimate the noise parameters from empirical local moments and then recover the memory kernel by RKHS regularization adapted to the data in the first moment equation.  The final subsection gives finite sample and conditional bounds for kernel recovery.

\subsection{Observation data and local moments}
\label{subsec:history_matched_observations}

We observe a reference velocity trajectory $\{v_l\}_{l=0}^{L}$ on the uniform grid $t_l=l\Delta t$ over $[0,T]$.
At each internal time $t_l$, the available data consist of an ensemble of $M$ increments over a short interval
\begin{equation}
\mathcal D_l
=
\left\{
\Delta v_{l,m}
\right\}_{m=1}^{M},
\qquad
\Delta v_{l,m}
=
v^{(m)}_{t_l+\Delta t}-v_l.
\label{eq:increment_ensemble}
\end{equation}
In the ideal setting, all increments in $\mathcal D_l$ share the complete reference history $\mathcal H_{t_l}^v=\{v_s:0\leq s\leq t_l\}$, and their future Brownian and
L\'evy increments are conditionally independent.  Such increment ensembles may be obtained by controlled repetition or numerical resampling; the inverse procedure requires the conditional ensembles rather than a particular acquisition mechanism.

Under an exact embedding, the augmented state $X_t=(v_t,\bm z_t)$ is determined by $\mathcal H_t^v$ and contains the history information needed for the future transition law. Thus, for every bounded measurable increment statistic $f$,
\begin{equation}
\mathbb E[f(\Delta v_h)\mid X_t]
=\mathbb E[f(\Delta v_h)\mid\mathcal H_t^v].
\label{eq:state_history_equivalence}
\end{equation}
This identity covers the truncated statistics used with L\'evy forcing and extends to integrable increment statistics.
The conditional moments can consequently be estimated from
the increment ensembles without reconstructing the auxiliary variables $\bm z_t$.
Departures from the condition of a common history are examined in Subsection~\ref{subsec:reference_conditioning_stress_test}.

\subsection{Estimation of noise parameters}
\label{subsec:stochastic_parameters}

\paragraph{Brownian--L\'evy parameters.}
The stable component is additive and state independent, so increments from different reference states can be pooled.
These increments are observations over a finite interval rather than direct samples of individual jumps and may contain drift, Brownian fluctuations, small jumps, or multiple jumps.
Threshold exceedances identify the contribution of large jumps only in the small-time regime underlying classical threshold methods~\cite{shimizu2006,mancini2009,aitsahalia2009testing,aitsahalia2009}.

For a threshold $\varepsilon>0$, define the pooled exceedance count
\begin{equation}
N_\varepsilon
=
\sum_{l=1}^{L-1}
\sum_{m=1}^{M}
\mathbf{1}_{\{
|\Delta v_{l,m}|\geq\varepsilon
\}}.
\label{eq:tail_count}
\end{equation}
The stable tail intensity $2c_\alpha\kappa^\alpha\varepsilon^{-\alpha}/\alpha$, together with the exceedance rate and magnitudes, gives the unbinned Poisson--Pareto estimators~\cite{hill1975,mittnik1996}
\begin{equation}
\widehat\alpha
=
N_\varepsilon
\left[\sum_{l=1}^{L-1}\sum_{m=1}^{M}
\mathbf{1}_{\{|\Delta v_{l,m}|\geq\varepsilon\}}
\log\!\left(
\frac{|\Delta v_{l,m}|}{\varepsilon}
\right)\right]^{-1},
\qquad
\widehat\kappa
=
\left[
\frac{
\widehat\alpha\,
\varepsilon^{\widehat\alpha}
N_\varepsilon
}{
2c_{\widehat\alpha}(L-1)M\Delta t
}
\right]^{1/\widehat\alpha}.
\label{eq:tail_parameter_estimators}
\end{equation}

The Poisson--Pareto estimator combines exceedance counts and magnitudes through the leading small-time jump law.
Asymptotic threshold sequences commonly satisfy $\varepsilon_{\Delta t}\to0$, $\varepsilon_{\Delta t}/\sqrt{\Delta t}\to\infty$, and $\Delta t\,\varepsilon_{\Delta t}^{-\alpha}\to0$ as $\Delta t\to0$ ~\cite{mancini2009,aitsahalia2009}.  For finite data, a positive $\varepsilon$ is prescribed before estimation and used in both the Poisson--Pareto estimator and the truncated Kramers--Moyal moments.
It should exceed the continuous increment scale while retaining enough exceedances for tail estimation.
The contribution of small jumps to the second moment is removed through $J_\varepsilon(\widehat\alpha,\widehat\kappa)$, while symmetry removes its leading contribution from the truncated first moment.

\paragraph{Brownian amplitude.}
In the Brownian case, we use the empirical second Kramers--Moyal coefficient
\begin{equation}
\widehat D_{B,l}^{(2)}
=
\frac{1}{2M\Delta t}
\sum_{m=1}^{M}(\Delta v_{l,m})^2
\approx\frac{1}{2}\sigma^2(v_l).
\end{equation}

In the Brownian--L\'evy case, after estimating $(\widehat\alpha,\widehat\kappa)$, the target for the Brownian amplitude at each state is obtained from the truncated second moment
\begin{equation}
\widehat D_{BL,l}^{(2)}
=
\frac{1}{2M\Delta t}
\sum_{m=1}^{M}
(\Delta v_{l,m})^2
\mathbf{1}_{\{|\Delta v_{l,m}|<{\varepsilon}\}}
-
\frac{1}{2}
J_{{\varepsilon}}
(\widehat\alpha,\widehat\kappa)
\approx
\frac{1}{2}\sigma^2(v_l).
\label{eq:tail_corrected_diffusion_target}
\end{equation}

For notational consistency, let $\widehat D_l^{(2)}$ denote $\widehat D_{B,l}^{(2)}$ in the Brownian case and $\widehat D_{BL,l}^{(2)}$ in the Brownian--L\'evy case.
Since the functional form of $\sigma(v)$ is unspecified, we approximate it by a neural network $\sigma_\theta(v)$ and estimate the parameters by minimizing
\begin{equation}
\widehat\theta
=
\arg\min_\theta
\frac{1}{L-1}
\sum_{l=1}^{L-1}
\left[
\widehat D_l^{(2)}
-
\frac{1}{2}\sigma_\theta^2(v_l)
\right]^2.
\label{eq:diffusion_loss}
\end{equation}

\subsection{Regularized memory kernel reconstruction}
\label{subsec:kernel_reconstruction}

The corrected second Kramers--Moyal coefficient is used above to estimate the Brownian amplitude. We now use the first coefficient to recover the memory
kernel.

For each interior index $l$, its empirical estimate is
\nopagebreak[4]
\begin{equation}
\widehat D_l^{(1)}
=
\begin{cases}
\displaystyle
\frac{1}{M\Delta t}
\sum_{m=1}^{M}\Delta v_{l,m},
& \kappa=0,
\\
\displaystyle
\frac{1}{M\Delta t}
\sum_{m=1}^{M}
\Delta v_{l,m}
\mathbf{1}_{\{|\Delta v_{l,m}|<{\varepsilon}\}},
& \kappa>0.
\end{cases}
\label{eq:empirical_first_km}
\end{equation}
The ordinary increment average is used when $\kappa=0$, and its truncated counterpart is used when $\kappa>0$. By the first Kramers--Moyal relation, subtracting this coefficient from the known force estimates the memory convolution
\begin{equation}
\widehat y_l
:=
F(v_l)-\widehat D_l^{(1)}
\approx
\int_0^{t_l}\gamma(t_l-s)v_s\,\mathrm{d}s,
\qquad
l=1,\ldots,L-1.
\label{eq:drift_data}
\end{equation}
The point $t_0$ is excluded because the memory integral vanishes there. 
Using a reconstruction basis $\{\phi_k\}_{k=1}^N$ and the approximation $\gamma(t)\approx\sum_{k=1}^Nc_k\phi_k(t)$ gives
\begin{equation}
\widehat y_l
\approx
\sum_{k=1}^{N}\Phi_{lk}c_k,
\qquad
\Phi_{lk}
=
\Delta t
\sum_{j=0}^{l-1}
\phi_k(t_l-t_j)v_j.
\label{eq:response_matrix}
\end{equation}
With $\widehat{\bm y}=(\widehat y_1,\ldots,\widehat y_{L-1})^\top$ and $\bm c=(c_1,\ldots,c_N)^\top$, the resulting system is $\bm\Phi\bm c\approx\widehat{\bm y}$.
Following the RKHS formulation adapted to operator data in~\cite{lu2022dartr}, we estimate $\bm c$ by
\begin{equation}
\widehat{\bm c}
=
\underset{\bm c\in\operatorname{range}(\bm\Phi^\top)}{\arg\min}
\left\{
\|\bm\Phi\bm c-\widehat{\bm y}\|_2^2
+
\lambda\bm c^\top(\bm\Phi^\top\bm\Phi)^\dagger\bm c
\right\},
\qquad
\left[
\bm\Phi^\top\bm\Phi
+\lambda(\bm\Phi^\top\bm\Phi)^\dagger
\right]
\widehat{\bm c}
=\bm\Phi^\top\widehat{\bm y}.
\label{eq:normal_equation}
\end{equation}
Here $\dagger$ denotes the Moore--Penrose inverse; {when $\bm\Phi$ has full column rank}, $(\bm\Phi^\top\bm\Phi)^\dagger=
(\bm\Phi^\top\bm\Phi)^{-1}$.
The parameter $\lambda$ is selected by the strong robust generalized cross-validation criterion (R1GCV)~\cite{lukas2008strong}.
Finally, the reconstructed memory kernel is given by
\begin{equation}
\widehat\gamma(t)
=
\sum_{k=1}^{N}
(\widehat{\bm c})_k\phi_k(t),
\qquad
0\leq t\leq T.
\label{eq:kernel_reconstruction}
\end{equation}
The complete estimation procedure is summarized in Algorithm~\ref{alg:algorithm_summary}.

\begin{algorithm}[!thb]
\caption{Two-stage estimation procedure}
\label{alg:algorithm_summary}
\begin{algorithmic}[1]

\REQUIRE
Velocity $\{v_l\}_{l=0}^{L}$; ensembles $\{\mathcal D_l\}_{l=1}^{L-1}$; time grid $\{t_l\}_{l=0}^{L}$.

\ENSURE
$\widehat\sigma(v)$ and $\widehat\gamma(t)$; for the Brownian--L\'evy case, $\widehat\alpha$ and $\widehat\kappa$.

\IF{Brownian--L\'evy forcing is considered}
    \STATE {Estimate $(\widehat\alpha,\widehat\kappa)$ at the prescribed threshold $\varepsilon$ using Eq.~\eqref{eq:tail_parameter_estimators}.}
    \STATE {Form the corrected targets for the Brownian amplitude using Eq.~\eqref{eq:tail_corrected_diffusion_target}.}
\ELSE
    \STATE Construct the empirical second Kramers--Moyal coefficients.
\ENDIF

\STATE Estimate $\widehat\sigma(v)=\sigma_{\widehat\theta}(v)$ by minimizing Eq.~\eqref{eq:diffusion_loss}.

\STATE Compute $\widehat D_l^{(1)}$ using Eq.~\eqref{eq:empirical_first_km}.

\STATE Assemble $\widehat{\bm y}$ and $\bm\Phi$ from Eqs.~\eqref{eq:drift_data} and~\eqref{eq:response_matrix}.

\STATE Select $\widehat\lambda>0$, solve Eq.~\eqref{eq:normal_equation}, and reconstruct $\widehat\gamma$ using Eq.~\eqref{eq:kernel_reconstruction}.

\end{algorithmic}
\end{algorithm}

\subsection{Identifiability and error analysis}
\label{sec:error_analysis}

We condition on the reference histories, denoted collectively by $\mathcal H$, and resample only the increment ensembles. Thus $\bm\Phi$ is fixed. Write $h=\Delta t$, $\ell=L-1$, and
\begin{equation}
b_l=F(v_l)-\int_0^{t_l}\gamma(t_l-s)v_s\,\mathrm{d}s,
\qquad \sigma_l=\sigma(v_l).
\label{eq:conditional_local_coefficients}
\end{equation}

\paragraph{Identifiable quantities.}
For $\kappa>0$, the infinitesimal exceedance rate is
$\Lambda_\varepsilon=2c_\alpha\kappa^\alpha\varepsilon^{-\alpha}/\alpha$.
The ratio $\Lambda_{q\varepsilon}/\Lambda_\varepsilon=q^{-\alpha}$ identifies $\alpha$, and the rate then identifies $\kappa$. The corrected second moment identifies $\sigma^2(v_l)$ and hence the nonnegative amplitude $\sigma(v_l)$. At $\kappa=0$, $\alpha$ is not identifiable. For the chosen reconstruction space, noiseless discrete responses identify all coefficients exactly when $\ker\bm\Phi=\{0\}$. A rank-deficient design identifies their projection onto $\operatorname{range}(\bm\Phi^\top)$.

\paragraph{Local statistical errors.}
Let $Y_{l,m}=\Delta v_{l,m}$ for Brownian forcing and
$Y_{l,m}=\Delta v_{l,m}\mathbf1_{\{|\Delta v_{l,m}|<\varepsilon\}}$ for Brownian--L\'evy forcing. For conditionally independent replicas with the same law within each row, define
\begin{equation}
m_{j,l}=h^{-1}\mathbb E[Y_{l,1}^{j}\mid\mathcal H],\qquad
V_{j,l}=h^{-1}\operatorname{Var}(Y_{l,1}^{j}\mid\mathcal H),\qquad j=1,2.
\label{eq:local_finite_lag_moments}
\end{equation}
Whenever $V_{j,l}<\infty$, the variance of $(Mh)^{-1}\sum_mY_{l,m}^{j}$ is $V_{j,l}/(Mh)$. This identity applies to exact nonlinear increments.

\begin{samepage}
For frozen Brownian increments $\Delta v_{l,m}=b_lh+\sigma_l\sqrt h\,\xi_{l,m}$ with conditionally independent $\xi_{l,m}\sim\mathcal N(0,1)$,
\begin{align}
\mathbb E[\widehat D_l^{(1)}\mid\mathcal H]&=b_l,
&\operatorname{Var}(\widehat D_l^{(1)}\mid\mathcal H)&=\frac{\sigma_l^2}{Mh},
\label{eq:d1_error_exact}\\
\mathbb E[\widehat D_l^{(2)}\mid\mathcal H]&=\frac{\sigma_l^2+b_l^2h}{2},
&\operatorname{Var}(\widehat D_l^{(2)}\mid\mathcal H)&=\frac{\sigma_l^4/2+b_l^2\sigma_l^2h}{M}.
\label{eq:d2_error_exact}
\end{align}
\end{samepage}

\begin{proposition}[Truncated moment errors]
\label{prop:levy_local_moments}
For Brownian--L\'evy forcing, fix $\varepsilon>0$ independently of the sampled increments. For $0<\delta<1$, put $x_\delta=\log(4\ell/\delta)$ and $B_1=2\varepsilon$, $B_2=\varepsilon^2$. With conditional probability at least $1-\delta$, simultaneously for $j=1,2$ and $l\le\ell$,
\begin{equation}
\left|\frac1{Mh}\sum_mY_{l,m}^{j}-m_{j,l}\right|
\le a_{j,l}:=\sqrt{\frac{2V_{j,l}x_\delta}{Mh}}+\frac{B_jx_\delta}{3Mh}.
\label{eq:levy_first_bernstein}
\end{equation}
Writing $\beta_{1,l}=|m_{1,l}-b_l|$ and
$\beta_{2,l}=|m_{2,l}-\sigma_l^2-J_\varepsilon(\alpha,\kappa)|$, the same event gives, for a tail fit with $0<\widehat\alpha<2$,
\begin{equation}
|\widehat D_l^{(1)}-b_l|\le a_{1,l}+\beta_{1,l},\qquad
\left|\widehat D_l^{(2)}-\frac{\sigma_l^2}{2}\right|
\le\frac{a_{2,l}+\beta_{2,l}+|\widehat J_\varepsilon-J_\varepsilon|}{2},
\qquad
\text{where }
\widehat J_\varepsilon=J_\varepsilon(\widehat\alpha,\widehat\kappa).
\label{eq:levy_diffusion_propagation}
\end{equation}

\end{proposition}
\begin{proof}
See Appendix~\ref{app:proofs}.
\end{proof}
These bounds hold for every $0<\alpha<2$. At each fixed matched history and threshold $\varepsilon$, the local identities give $\beta_{j,l}\to0$ as $h\downarrow0$. In particular,
\begin{equation}
V_{1,l}\longrightarrow\sigma_l^2+J_\varepsilon(\alpha,\kappa),\qquad
V_{2,l}\longrightarrow\frac{2c_\alpha\kappa^\alpha}{4-\alpha}\varepsilon^{4-\alpha}.
\label{eq:levy_variance_limits}
\end{equation}
Thus the truncated second statistic also has sampling fluctuations of order $(Mh)^{-1/2}$ at fixed threshold, unlike the frozen Brownian second statistic.

For the actual tail estimators, let $\widehat\Lambda_\varepsilon=N_\varepsilon/(\ell Mh)$. Substitution in Eq.~\eqref{eq:tail_parameter_estimators} gives the identity
\begin{equation}
\widehat J_\varepsilon-J_\varepsilon
=\varepsilon^2\left[
\frac{\widehat\alpha}{2-\widehat\alpha}(\widehat\Lambda_\varepsilon-\Lambda_\varepsilon)
+\frac{2\Lambda_\varepsilon(\widehat\alpha-\alpha)}{(2-\widehat\alpha)(2-\alpha)}\right].
\label{eq:stable_correction_sensitivity}
\end{equation}
This identity propagates count and stability-index errors to the small-jump correction for $0<\widehat\alpha<2$. The count error is bounded in Appendix~\ref{app:proofs}; the denominators exhibit the sensitivity as $\alpha$ approaches $2$. For prescribed $\varepsilon$, neither the tail fit nor the Brownian-amplitude fit enters the first-moment kernel estimator.

For the amplitude regression, set $f_\theta=\sigma_\theta^2/2$, $f_\star=\sigma^2/2$, and $\|\bm a\|_\ell^2=\ell^{-1}\sum_la_l^2$. If the achieved residual norm is within $\eta_{\rm opt}$ of its infimum, then
\begin{equation}
\|f_{\widehat\theta}(\bm v)-f_\star(\bm v)\|_\ell
\le\inf_\theta\|f_\theta(\bm v)-f_\star(\bm v)\|_\ell
+2\|\widehat{\bm D}^{(2)}-f_\star(\bm v)\|_\ell+\eta_{\rm opt}.
\label{eq:diffusion_fit_decomposition}
\end{equation}
Moreover, $\|\widehat\sigma(\bm v)-\sigma(\bm v)\|_\ell^2
\le2\|f_{\widehat\theta}(\bm v)-f_\star(\bm v)\|_\ell$ for nonnegative amplitudes. This bounds the Brownian-amplitude error at the sampled states.

\paragraph{Propagation to the memory kernel.}
Assume $\gamma\in L^2(0,T)$ and linearly independent continuous reconstruction functions. Let $\gamma_N=\bm\phi^\top\bm c_N$ be the $L^2$ projection of $\gamma$, and set $\eta_N=\|\gamma-\gamma_N\|_{L^2}$. Define
\begin{equation}
\bm G_N=\int_0^T\bm\phi(t)\bm\phi(t)^\top\,\mathrm{d}t,\qquad
 g_N^+=\lambda_{\max}(\bm G_N),\quad g_N^-=\lambda_{\min}(\bm G_N).
\label{eq:l2_basis_gram}
\end{equation}
The resolution index $\mu_N=\min_{\bm c\ne0}\|\bm\Phi\bm c\|_\ell/\|\bm\phi^\top\bm c\|_{L^2}$ is positive exactly at full column rank and measures identifiability independently of basis scaling.
Let $\bm\Phi=\bm U_r\operatorname{diag}(s_i)\bm V_r^\top$ be its compact SVD, with rank $r\ge1$ and $s_i>0$. Put $\bm\Pi=\bm V_r\bm V_r^\top$ and
\begin{align}
\rho_{N,\lambda}&=\max_{i\le r}\frac{s_i^3}{s_i^4+\lambda},\qquad
\bm R_\lambda=\bm I-\bm\Pi+\bm V_r\operatorname{diag}\!\left(\frac{\lambda}{s_i^4+\lambda}\right)\bm V_r^\top,
\nonumber\\
\mathcal B_{N,\lambda}&=\|\bm G_N^{1/2}\bm R_\lambda\bm c_N\|_2.
\label{eq:rkhs_filter_factors}
\end{align}
The last term retains both regularization bias and the unresolved coefficient component. For $r=N$, it is at most
$\sqrt{g_N^+/g_N^-}\,\lambda\|\gamma\|_{L^2}/(\min_i s_i^4+\lambda)$.

Set $\beta_{1,h}=\max_l|m_{1,l}-b_l|$ for the actual sampling law and
\begin{equation}
q_{N,h}=\max_l\left|\int_0^{t_l}\gamma_N(t_l-s)v_s\,\mathrm{d}s-(\bm\Phi\bm c_N)_l\right|,
\qquad r_{N,h}=\beta_{1,h}+\|v\|_{L^2}\eta_N+q_{N,h}.
\label{eq:deterministic_residual_bound}
\end{equation}
Appendix~\ref{app:proofs} separates finite-lag, history-preparation, and numerical contributions to $\beta_{1,h}$ and bounds $q_{N,h}$ directly for c\`adl\`ag paths, including infinite-variation L\'evy paths.

\begin{theorem}[Conditional kernel recovery]
\label{thm:resolved_recovery}
Condition on the reference histories and suppose the replicas are independent over all $(l,m)$. Let $v$ be c\`adl\`ag on $[0,T]$ and $\overline V_1=\max_lV_{1,l}<\infty$. Use the minimum Euclidean norm solution of Eq.~\eqref{eq:normal_equation}. For $0<\delta<1$, define
\begin{equation}
\mathfrak z_\delta=
\begin{cases}
\displaystyle\sqrt{\frac{r\overline V_1}{Mh\delta}},&\kappa=0,\\
\displaystyle2\sqrt{\frac{2\overline V_1x_{r,\delta}}{Mh}}+
\frac{4\varepsilon x_{r,\delta}}{3Mh},&\kappa>0,
\end{cases}
\qquad x_{r,\delta}=r\log5+\log(2/\delta).
\label{eq:projected_sampling_bound}
\end{equation}
Then, with conditional probability at least $1-\delta$, simultaneously for every $\lambda>0$,
\begin{equation}
\|\widehat\gamma_\lambda-\gamma\|_{L^2}
\le\eta_N+\sqrt{g_N^+}\,\rho_{N,\lambda}
\left(\mathfrak z_\delta+\sqrt\ell\,r_{N,h}\right)+\mathcal B_{N,\lambda}.
\label{eq:resolved_kernel_bound}
\end{equation}
The same event covers any positive $\widehat\lambda$ selected from the data, including R1GCV.
\end{theorem}
\begin{proof}
See Appendix~\ref{app:proofs}.
\end{proof}
The Brownian bound uses only conditional variance; Appendix~\ref{app:proofs} gives its Gaussian refinement for frozen steps. This quantifies stability at the selected parameter. Increasing $M$ reduces sampling error, while deterministic discrepancies and unresolved directions can leave an error floor. Response and design perturbations contribute $\Delta\bm y-\Delta\bm\Phi\bm c_N$ to $\bm d$; the spectral bound then uses the matrix actually inverted.

\section{Numerical experiments}
\label{sec:numerical_experiments}

This section reports numerical results for Brownian and Brownian--L\'evy models.  We first present representative reconstructions and comparisons with typical kernel estimators, and then examine sensitivity to sampling, observation noise, force and kernel forms, and departures from exact history matching.
The code for the numerical experiments is available at \url{https://github.com/yeungineer/gle-inverse}.

\subsection{Representative inverse problems}
\label{subsec:representative_inverse_problems}

We consider a common reference problem with $F=0$ and a memory kernel whose parameters are randomly generated subject to exponential decay,
\begin{equation}
\begin{split}
\gamma(t)={}&0.4021e^{(-0.1425-0.2987\mathrm{i})t}
+0.4021e^{(-0.1425+0.2987\mathrm{i})t}\\
&+0.2753e^{-0.7814t}+0.4638e^{-0.9643t}
+0.3187e^{-0.4012t}.
\end{split}
\label{eq:numerical_reference_kernel}
\end{equation}

Reference trajectories are generated by the Euler--Maruyama scheme with $\Delta t=10^{-2}$, $L=3\times10^3$, and $v_0\sim\mathcal N(0,1)$. The numerical update is
\begin{equation}
v_{n+1}
=v_n+\left[
F(v_n)-\Delta t\sum_{j=0}^{n-1}\gamma(t_n-t_j)v_j
\right]\Delta t
+\sigma(v_n)\sqrt{\Delta t}\,\xi_n
+\kappa\Delta L_n^{(\alpha)},
\label{eq:numerical_em_scheme}
\end{equation}
for $n=0,\ldots,L-1$, where
$\xi_n\stackrel{\mathrm{i.i.d.}}{\sim}\mathcal N(0,1)$ and
$\Delta L_n^{(\alpha)}=L_{t_{n+1}}^{(\alpha)}-L_{t_n}^{(\alpha)}$.
The stable increments are sampled by the Chambers--Mallows--Stuck method~\cite{chambers1976method}. We consider a Brownian model with $\sigma(v)=0.5$ and $\kappa=0$, and a Brownian--L\'evy model with $\sigma(v)=0.5(1+\cos v)$, $\alpha=1.6$, and $\kappa=0.12$.
At each reference time $t_l$, the stored history is held fixed and one step of Eq.~\eqref{eq:numerical_em_scheme} is repeated with independent noise to form the increment ensemble in Subsection~\ref{subsec:history_matched_observations}.

\begin{figure}[tb]
\centering
\includegraphics[width=\textwidth]{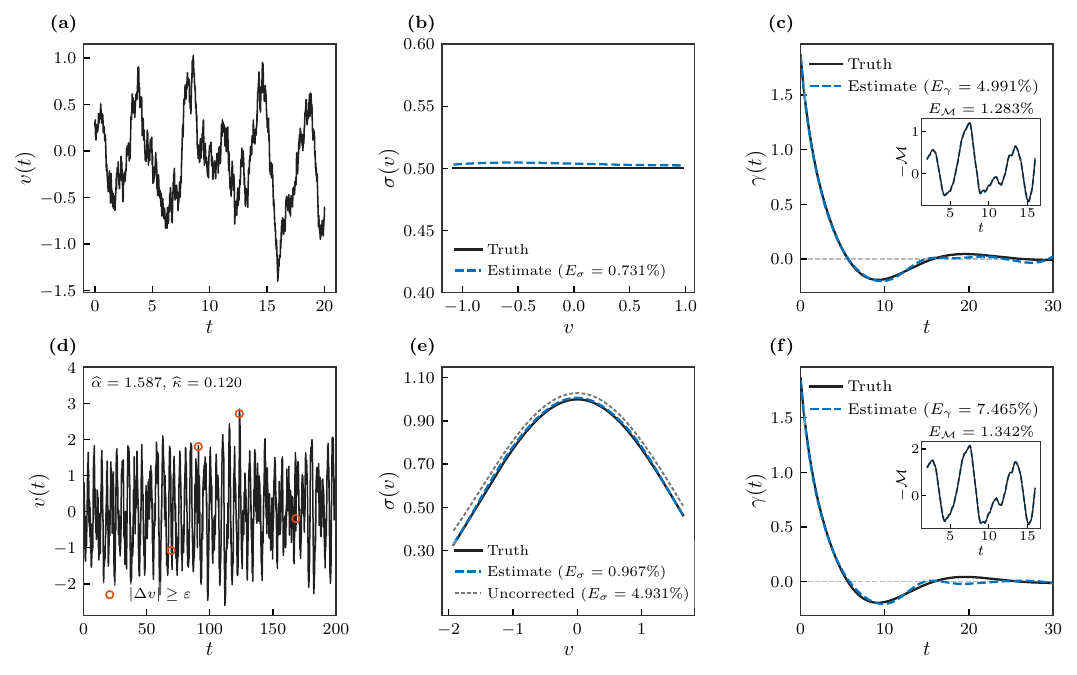}
\caption{Representative joint reconstructions from increment ensembles with matched histories.
The upper row corresponds to the Brownian case, and the lower row corresponds to the Brownian--L\'evy case.
Panels (a,d) show the reference velocities, with circles in (d) marking increments with magnitude at least $\varepsilon$. Panels (b) and (e) compare the true and estimated Brownian amplitudes; the gray curve in (e) omits the correction for small jumps. Panels (c) and (f) compare the true and estimated memory kernels, with the corresponding memory forces shown in the insets.}
\label{fig:representative_recovery}
\end{figure}

The Brownian--L\'evy experiments use the prescribed threshold
$\varepsilon=1$. The regression for the Brownian amplitude in
\eqref{eq:diffusion_loss} uses three hidden layers of width $64$, GELU activations, a softplus output, AdamW, and an L-BFGS refinement. Let $\mathcal M(t)=\int_0^t\gamma(t-s)v_s\,\mathrm{d}s$ denote the memory convolution.  The relative errors are
\begin{equation}
{
E_f
=
\frac{\|\widehat f-f\|_{L^2}}{\|f\|_{L^2}},
\qquad
f\in\{\sigma,\gamma,\mathcal M\},}
\qquad
E_{\alpha,\kappa}
=
\left(
|\widehat\alpha-\alpha|^2+
|\log(\widehat\kappa/\kappa)|^2
\right)^{1/2}.
\label{eq:numerical_error_metrics}
\end{equation}

The reconstruction uses $N=12$ cubic B-splines on $[0,30]$. Figure~\ref{fig:representative_recovery} summarizes the reconstruction results for $M=2000$. In the Brownian case, the selected regularization parameter $\lambda=0.6953$ gives $E_\sigma=0.0073$, $E_\gamma=0.0499$, and $E_{\mathcal M}=0.0128$. 
In the Brownian--L\'evy case, $\widehat\alpha=1.5867$ and $\widehat\kappa=0.1200$. The regularization parameter is $\lambda=6.3427$, and the corresponding errors are $E_\sigma=0.0097$, $E_\gamma=0.0746$, and $E_{\mathcal M}=0.0134$.
Removing the correction for small jumps in Eq.~\eqref{eq:tail_corrected_diffusion_target} increases the Brownian amplitude error from $0.0097$ to $0.0493$. The retained jumps therefore contribute non-negligibly to the truncated quadratic variation and, without the correction, bias the Brownian amplitude estimate.
In both representative cases, the discrete Volterra matrix has full column rank. The RKHS penalty adapted to the data in Eq.~\eqref{eq:normal_equation} regularizes coefficient directions that are weakly resolved by the observed history.
The memory force error is also smaller than the kernel error in both cases. This is consistent with the smoothing action of the Volterra convolution, under which kernel errors can have a weaker effect on the resulting memory force. Accordingly, $E_{\mathcal M}$ serves as a complementary measure of the dynamical accuracy of the reconstructed kernel.

\subsection{Comparison with typical kernel estimators}
\label{subsec:benchmark_comparison}

We compare the proposed estimator with the Prony--Sobolev method of Lang and Lu~\cite{lang2026memory} and the likelihood--EM method for GLEs in~\cite{vroylandt2022likelihood}.
The benchmark uses additive Gaussian forcing, $\Delta t=10^{-2}$, $M=1000$, $L=3000$, and $B=ML=3\times10^6$. The proposed estimator uses increment ensembles conditioned on history, Prony--Sobolev uses ensemble correlations from a common initial condition, and likelihood--EM uses independent trajectories.
Figure~\ref{fig:benchmark_comparison}(a) gives a kernel error of $0.0602$ for the proposed estimator, $0.0914$ for Prony--Sobolev, and $0.1628$ for likelihood--EM.  The EM run reaches the prescribed limit of $2000$ iterations.
\begin{figure}[tb]
\centering
\includegraphics[width=\textwidth]{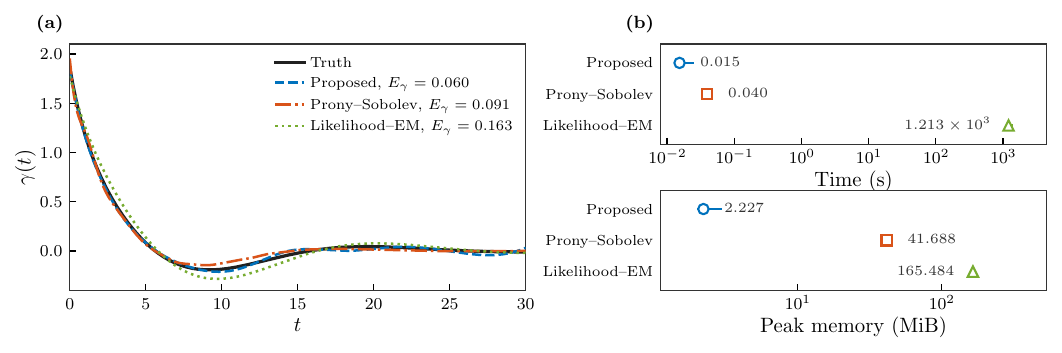}
\caption{Comparison of kernel estimators under additive Gaussian forcing with the common nominal data volume $B=3\times10^6$.
Panel (a) shows the reference kernel and the proposed, Prony--Sobolev, and likelihood--EM estimates, with their relative $L^2$ errors reported in the legend. Panel (b) reports kernel fitting time and incremental peak memory after input preparation; markers and horizontal bars denote medians and IQRs over six repetitions.}
\label{fig:benchmark_comparison}
\end{figure}

Panel (b) measures kernel fitting after the inputs have been prepared. Over six repetitions, the median time and incremental peak memory are $0.0153$~s and $2.23$~MiB for the proposed estimator, $0.0396$~s and $41.7$~MiB for Prony--Sobolev, and $1.21\times10^3$~s and $165.5$~MiB for likelihood--EM.
The differences in fitting cost reflect the form of the three estimators. The proposed method estimates the conditional moments and solves one regularized linear system, Prony--Sobolev reconstructs the kernel from ensemble correlations, and likelihood--EM repeatedly updates the latent variables and model parameters. In this experiment, the direct conditional-moment formulation gives the smallest kernel error together with the lowest fitting time and memory use.

\subsection{Sampling and observation sensitivity}
\label{subsec:parameter_studies}

We next examine how the amount and quality of the observed data affect the reconstruction. The increment ensemble size controls the accuracy of the empirical conditional moments, while the observation length determines the number of equations available for the Volterra inversion.  We also add measurement noise under several observation settings to assess its effect on the stochastic parameters and the memory kernel.

\paragraph{Increment ensemble size and observation length.}

In Figure~\ref{fig:sampling_budget}(a), the normalized RMSEs of both empirical Kramers--Moyal coefficients decrease from approximately $0.0317$ at $M=1000$ to $0.0119$ at $M=7000$, following the $M^{-1/2}$ reference line. Each RMSE is divided by the RMS, across reference states, of the standard deviations in Eqs.~\eqref{eq:d1_error_exact} and~\eqref{eq:d2_error_exact} evaluated at $M=1$. Errors are centered at the frozen-step expectations, including the $b_l^2h/2$ contribution for the second coefficient; here $b_l$ is the numerical drift used to generate the increments. The normalization is therefore independent of $M$.
At fixed $L=3000$, panel (b) shows a slower decrease in the median kernel error, from $0.0711$ to $0.0214$.
Panel (c) compares three observation lengths. At $M=500$, the median kernel error decreases from $0.0769$ for $L=3000$ to $0.0271$ for $L=6000$ and $0.0185$ for $L=9000$. For $L=9000$, increasing $M$ from $750$ to $1000$ changes the median only from $0.0117$ to $0.0108$. Thus increasing $M$ improves the conditional moments, while increasing $L$ adds equations to the Volterra system; the remaining kernel error is also governed by the conditioning and regularization of the inversion.
\begin{figure}[t]
\centering
\includegraphics[width=\textwidth]{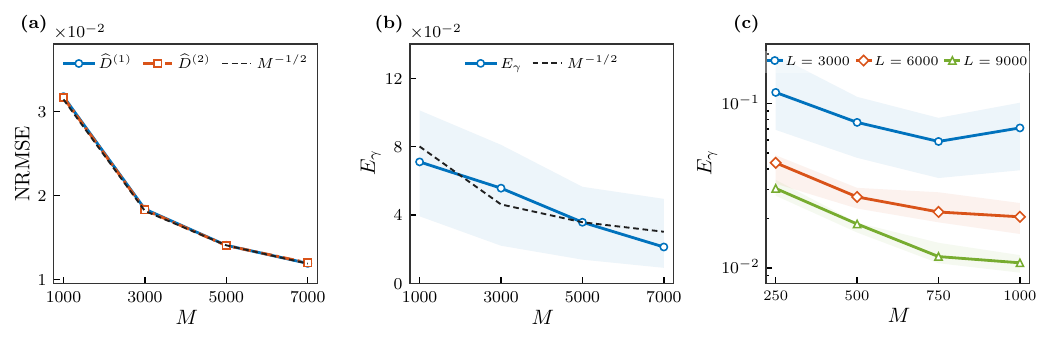}
\caption{Effects of the increment ensemble size and observation length in the Brownian case.
Panel (a) shows the normalized RMSEs of the empirical first and second Kramers--Moyal coefficients against $M$, together with the $M^{-1/2}$ reference rate. Panel (b) shows the kernel error against $M$ at $L=3000$, and panel (c) compares the kernel errors for three observation lengths. Solid curves denote medians, and the shaded bands in (b) and (c) denote IQRs over six realizations.}
\label{fig:sampling_budget}
\end{figure}

\paragraph{Observation noise.}
Figure~\ref{fig:observation_noise} examines independent centered Gaussian measurement noise with standard deviation $\sigma_{\rm obs}$ at the final point, at both endpoints, or along the full reference trajectory. In the first two settings, the observed increment is $\Delta v+\eta_+$ or $\Delta v+\eta_+-\eta_-$, while the clean history is used in $\bm\Phi$.For the untruncated Brownian estimators, these perturbations leave the first conditional moment unbiased and add $\sigma_{\rm obs}^2/(2\Delta t)$ and $\sigma_{\rm obs}^2/\Delta t$, respectively, to the second coefficient. With L\'evy truncation, the perturbation also includes changes in threshold membership.
Noise along the full trajectory also contaminates every row of $\bm\Phi$, and the Volterra inversion amplifies this errors-in-variables perturbation. In the Brownian--L\'evy case, the threshold $\varepsilon=1$ remains above the observation-noise levels, so the median tail-parameter errors change only mildly. Their wider IQRs reflect the $646$--$725$ exceedances available to the tail fit among approximately $6\times10^6$ increments.
At $\sigma_{\rm obs}=0.1$, the median Brownian amplitude errors for final-point, both-endpoint, and full-trajectory noise are $1.240$, $2.004$, and $2.009$. The corresponding kernel errors are $(0.0971,0.1554,0.7370)$ in the Brownian case and $(0.0676,0.0762,0.7474)$ in the Brownian--L\'evy case. For the latter, the median errors in $\alpha$ are $(0.067,0.064,0.069)$ and those in $\log\kappa$ are $(0.173,0.174,0.158)$.
\begin{figure}[t]
\centering
\includegraphics[width=\textwidth]{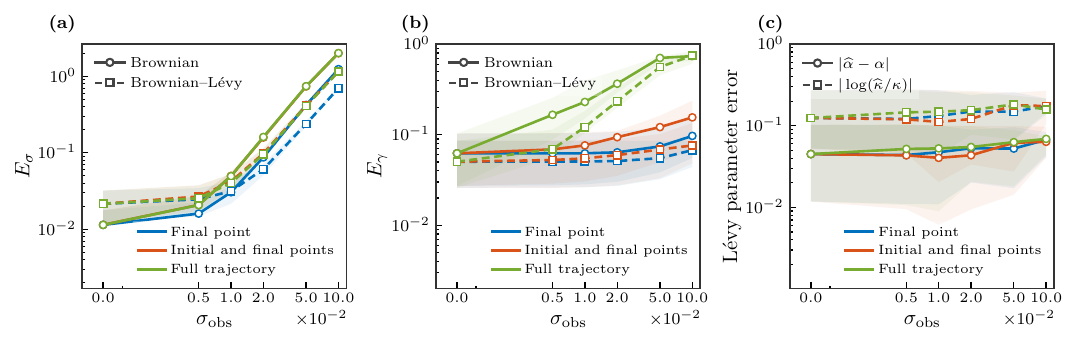}
\caption{Sensitivity to additive observation noise applied at the final point, at both interval endpoints, or along the full reference trajectory.
Panels (a) and (b) show the errors of the Brownian amplitude and memory kernel for the Brownian and Brownian--L\'evy models. Panel (c) shows $|\widehat\alpha-\alpha|$ and $|\log(\widehat\kappa/\kappa)|$ for the Brownian--L\'evy model. Curves and shaded bands denote medians and IQRs over six realizations.}
\label{fig:observation_noise}
\end{figure}

\subsection{Different dynamics and memory kernels}
\label{subsec:generality_experiments}

The representative problems above use a fixed force, a single set of Brownian--L\'evy parameters, and a rapidly decaying memory kernel. We now vary these model components to examine the reconstruction under different dynamics. The first experiment changes the known force and the stable parameters.
The second replaces the reference kernel with a sign-changing, algebraically decaying kernel.

\begin{figure}[tb]
\centering
\includegraphics[width=\textwidth]{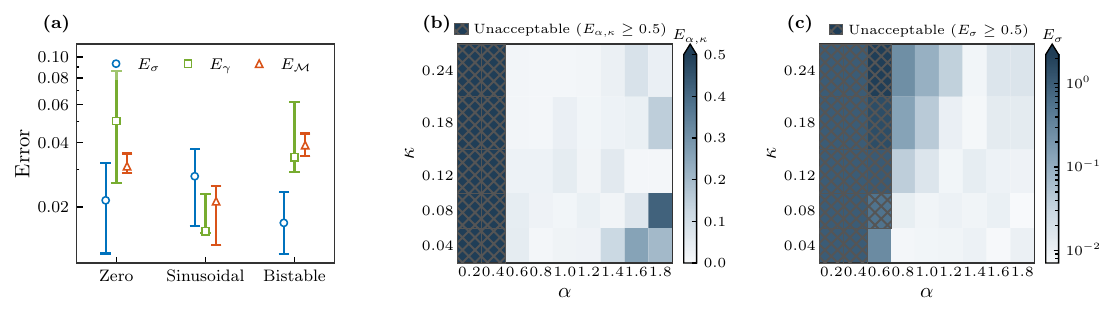}
\caption{Sensitivity to the force and L\'evy parameters. 
Panel (a) shows the medians and IQRs of $E_\sigma$, $E_\gamma$, and $E_{\mathcal M}$ over six realizations for the zero, sinusoidal, and bistable forces. Panels (b) and (c) show, for one representative realization, the joint error in the stable parameters and error in the Brownian amplitude over the $(\alpha,\kappa)$ grid. The color scale in (b) is capped at $0.5$, and crosshatching marks cells with $E_{\alpha,\kappa}\geq0.5$ in (b) or $E_\sigma\geq0.5$ in (c).}
\label{fig:force_and_levy_range}
\end{figure}

\paragraph{Forces and range of L\'evy parameters.}

Figure~\ref{fig:force_and_levy_range}(a) compares $F(v)=0$, $F(v)=\sin(2v)$, and $F(v)=v-v^3$ in the Brownian--L\'evy model.  Across the three forces, the largest median errors are $0.0279$ for $\sigma$, $0.0504$ for $\gamma$, and $0.0389$ for the memory force.
Panels (b) and (c) use $M=2000$, $L=3000$, and the grid $0.2\leq\alpha\leq1.8$, $0.04\leq\kappa\leq0.24$.  Of the $45$ parameter pairs, $35$ give $E_{\alpha,\kappa}<0.5$; the median errors over these cells are $E_{\alpha,\kappa}=0.0284$ and $E_\sigma=0.0118$.  All cells with $\alpha=0.2$ or $0.4$ exceed the cutoff for the joint error.
Their exceedance counts range from $1.28\times10^6$ to $5.65\times10^6$, while the L\'evy tail measure predicts only $1.37\times10^4$ to $4.07\times10^4$ jumps above the same threshold over all observation intervals.  The excess cannot be explained by the jump tail alone.
For these very heavy-tailed reference paths, the conditional drift location can be comparable with the threshold, so the pooled exceedances are no longer a sparse sample of the symmetric jump tail.
The Poisson--Pareto fit then produces large scale estimates and $E_\sigma\approx1$.  At $\alpha=0.6$, the {error in the tail parameters} is already between $0.0148$ and $0.0504$, while $E_\sigma$ ranges from $0.282$ to $2.24$. This intermediate row shows a different limitation.
The relative error in the estimated small-jump correction is only $0.0094$--$0.0195$ along this row, so the increase in $E_\sigma$ is not caused by inaccurate tail parameters. Instead, as $\kappa$ increases, the finite-step truncated second moment contains a larger jump contribution, and separating the Brownian component from this corrected moment becomes less accurate.  For $\alpha\geq0.8$, the median errors are $E_{\alpha,\kappa}=0.0287$ and $E_\sigma=0.0106$. The largest joint error in this region occurs at high $\alpha$ and small $\kappa$, where only a few increments exceed the threshold and tail estimation has greater sampling variation.

\begin{figure}[tb]
\centering
\includegraphics[width=\textwidth]{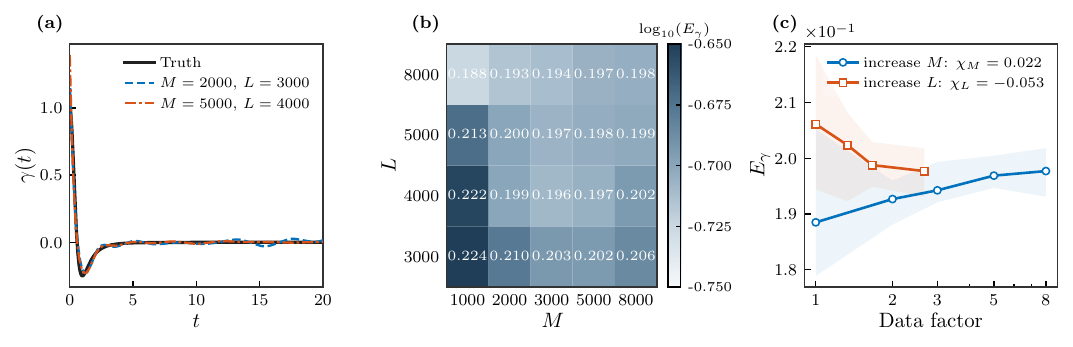}
\caption{Recovery of the algebraically decaying kernel.  
Panel (a) shows the true kernel and two representative reconstructions.  Panel (b) reports the median kernel error over six realizations for each $(M,L)$ pair, with color indicating $\log_{10}(E_\gamma)$.  Panel (c) shows the edge slices at $L=8000$ and $M=8000$; curves and bands denote medians and IQRs, while $\chi_M$ and $\chi_L$ are the corresponding median slopes on logarithmic axes.}
\label{fig:algebraic_kernel}
\end{figure}

\paragraph{Algebraically decaying memory.}

To examine memory structures beyond the rapidly decaying reference kernel in~\eqref{eq:numerical_reference_kernel}, we consider the following algebraically decaying kernel, which changes sign.
\begin{equation}
\gamma_{\mathrm{alg}}(t)
=
\frac{1-3t^2}{(1+t^2)^3},
\label{eq:algebraic_kernel}
\end{equation}
which satisfies $\gamma_{\mathrm{alg}}(t)\sim-3t^{-4}$ for large $t$ and has zero total mass on $[0,\infty)$. The Brownian model with zero force is retained.  The inverse uses $N=20$ cubic B-splines on $[0,30]$ with knots concentrated near the origin.
Both the forward convolution and $\bm\Phi$ use the left rectangle rule, and the coefficients are regularized by the RKHS term adapted to the data in Eq.~\eqref{eq:normal_equation}.
Figure~\ref{fig:algebraic_kernel}(a) shows two representative reconstructions with $E_\gamma=0.2611$ for $(M,L)=(2000,3000)$ and $E_\gamma=0.1984$ for $(M,L)=(5000,4000)$.  The reconstructions retain the initial positive lobe, the zero crossing, and the negative lobe.
Over the full grid in panel (b), the median errors range from $0.1885$ to $0.2237$. For $M\geq3000$ and $L\geq4000$, they remain between $0.1942$ and $0.2023$. For each realization, $\chi_M$ is the least-squares slope of $\log E_\gamma$ against $\log(M/M_{\min})$ along $L=8000$, and $\chi_L$ is defined analogously from $\log E_\gamma$ against $\log(L/L_{\min})$ along $M=8000$.  Panel (c) reports the medians of these slopes over the finite grid, $\chi_M=0.022$ and $\chi_L=-0.053$.  The corresponding median errors change from $0.1885$ to $0.1977$ as $M$ increases and from $0.2061$ to $0.1977$ as $L$ increases.
Thus, the reconstruction error is nearly flat over the larger
values of $M$ and $L$ in this grid.

\subsection{Approximate reference conditioning}
\label{subsec:reference_conditioning_stress_test}

The preceding experiments assume increment ensembles conditioned on a common reference history. Figure~\ref{fig:nonideal_conditioning} compares three history-preparation protocols for the Brownian and Brownian--L\'evy models to examine sensitivity to approximate history conditioning. In this experiment, both models use the state-dependent Brownian amplitude $\sigma(v)=0.5(1+\cos v)$.
Matched-history release generates increments directly from the prescribed reference history and serves as the ideal benchmark. Short-lag release permits a brief evolution between history matching and observation. Feedback-prepared release uses controlled trajectories to follow the reference history before the control is removed. The latter two protocols represent practical approximations to exact conditioning and test whether the Brownian amplitude, memory kernel, and resulting memory force are affected by imperfectly prepared histories.

\begin{figure}[t]
\centering
\includegraphics[width=\textwidth]{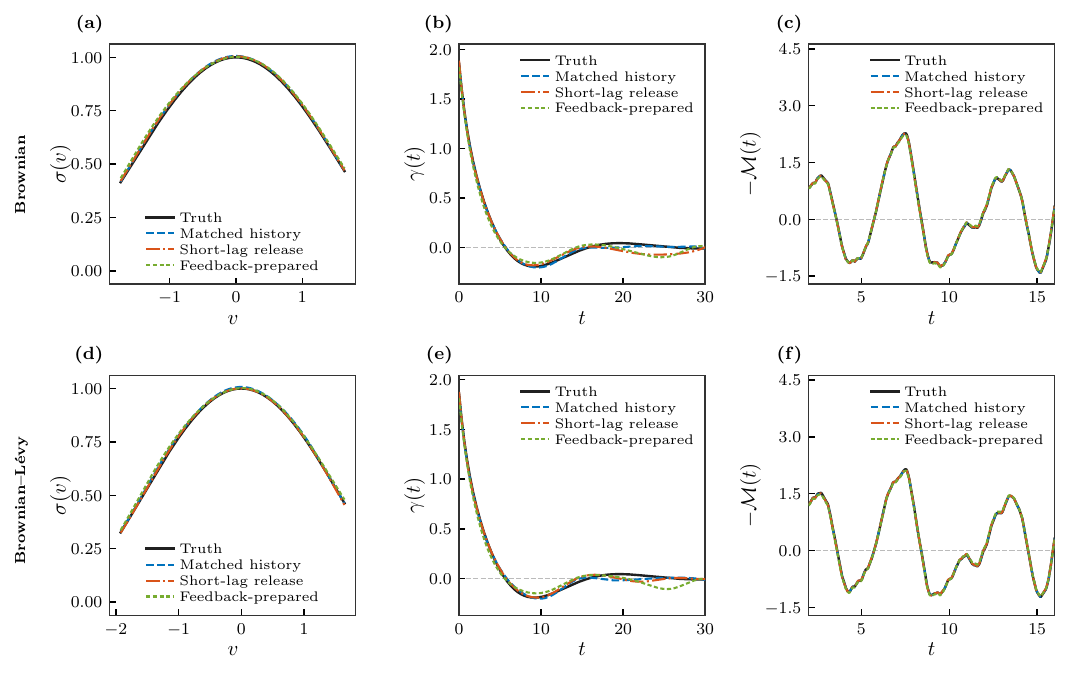}
\caption{Single-realization comparison of three history-to-release protocols. The first and second rows show the Brownian and Brownian--L\'evy models, respectively; the columns show the Brownian amplitude, memory kernel, and memory force $-\mathcal M$. Each panel compares the truth with matched-history, short-lag, and feedback-prepared releases.}
\label{fig:nonideal_conditioning}
\end{figure}

Across the three protocols, Brownian-amplitude errors remain small in both models, with $E_\sigma$ between $0.0046$ and $0.0178$, while $E_{\mathcal M}$ remains between $0.0111$ and $0.0315$. The kernel is more sensitive to history preparation. For the Brownian model, $E_\gamma$ increases from $0.0502$ for matched-history release to $0.1329$ for short-lag release and $0.1503$ for feedback-prepared release. The corresponding values for the Brownian--L\'evy model are $0.0746$, $0.0777$, and $0.1495$. Thus feedback preparation gives the largest kernel error in both cases, whereas the effect of a short lag depends more strongly on how the resulting history mismatch interacts with the realized trajectory.

This separation reflects the different roles of the two inverse stages. The Brownian amplitude is determined by the leading short-time quadratic variation, for which a small history mismatch primarily perturbs the drift and has only a higher-order influence on the local variance. Kernel recovery instead uses the conditional first moment through a first-kind Volterra equation, so bias caused by imperfect conditioning can be amplified along weakly resolved directions of the inverse problem. Nevertheless, the memory-force error remains substantially smaller than the kernel error because the convolution with the observed velocity smooths kernel discrepancies and suppresses components that are weakly excited by that trajectory. Accurate recovery of $-\mathcal M$ therefore indicates that the dynamically relevant action of the kernel is retained along the observed history, even when its full $L^2$ reconstruction is less accurate.

For the Brownian--L\'evy model, matched-history release gives $\widehat\alpha=1.5867$ and $\widehat\kappa=0.1200$; short-lag release gives $\widehat\alpha=1.6715$ and $\widehat\kappa=0.1395$; feedback-prepared release gives $\widehat\alpha=1.6069$ and $\widehat\kappa=0.1266$. The short-lag release gives the largest stability-index error in this realization.

\section{Conclusions}
\label{sec:conclusions}

We develop a data-driven framework for identifying memory and stochastic forcing in a one-dimensional generalized Langevin equation with state-dependent Brownian noise and optional additive symmetric $\alpha$-stable jumps. A finite-dimensional Markov embedding provides the conditional moment relations needed to separate memory drift, Brownian diffusion, and jump contributions.  Estimates of the stochastic parameters are then combined with an RKHS regularization of the Volterra equation for the memory kernel. The resulting procedure does not require equilibrium correlations or a fluctuation--dissipation relation. The analysis gives conditional finite-sample bounds for the local moment estimators and a conditional kernel recovery bound that accounts for sampling, finite lag, numerical integration, history preparation, and approximation in the reconstruction space.

The numerical results support the joint recovery of the stochastic parameters, memory kernel, and memory force for representative Brownian and Brownian--L\'evy models.  Increasing the increment ensemble size improves the empirical conditional moments at the expected sampling rate, but kernel accuracy also depends on the observation length and the resolution of the Volterra problem. This distinction is visible for the algebraically decaying kernel, where the error remains nearly unchanged over the larger data sizes considered.  Observation noise along the reference trajectory, inexactly prepared histories, and very heavy L\'evy forcing produce the main losses of accuracy in the experiments.

Extensions include multidimensional dynamics with matrix memory kernels, joint estimation of the deterministic force, and multiplicative L\'evy forcing. Adaptive reconstruction bases and alternative ways of obtaining history-conditioned increment ensembles are further directions.



\section*{Acknowledgements}
This work was supported by the National Natural Science Foundation of China (12401233 and 12501716), the NSFC International Creative Research Team (W2541005), the National Key R\&D Program of China (2021ZD0201300), the Guangdong Provincial Key Laboratory of Mathematical and Neural Dynamic Systems (2024B1212010004), the Guangdong Major Project of Basic Research (2025B0303000003), the Interdisciplinary Research Program of HUST (2023JCYJ012), the Hubei Provincial Natural Science Foundation of China (2026AFA034 and 2025AFB172), the Natural Science Foundation of Shenzhen (JCYJ20250604183428039), the Guangdong Basic and Applied Basic Research Foundation (2026A1515030054), and the Hubei Key Laboratory of Engineering Modeling and Scientific Computing.

\appendix
\section{Proofs}
\label{app:proofs}

\paragraph{Brownian moment identities.}
For $Y=b_lh+\sigma_l\sqrt h\,\xi$ with $\xi\sim\mathcal N(0,1)$, Gaussian moments give
\begin{align*}
\mathbb EY&=b_lh,& \operatorname{Var}(Y)&=\sigma_l^2h,\\
\mathbb EY^2&=b_l^2h^2+\sigma_l^2h,&
\operatorname{Var}(Y^2)&=4b_l^2\sigma_l^2h^3+2\sigma_l^4h^2.
\end{align*}
Independence and the normalizations $1/(Mh)$ and $1/(2Mh)$ give the stated identities for the Brownian estimators.

\begin{proof}[Proof of Proposition~\ref{prop:levy_local_moments}]
The centered variables $Y_{l,m}^j-\mathbb E[Y_{l,m}^j\mid\mathcal H]$ have absolute value at most $B_j$ and total variance $MhV_{j,l}$ over a row. The scalar Bernstein inequality gives
\begin{equation*}
\mathbb P\!\left(\left|\sum_m(Y_{l,m}^j-\mathbb E[Y_{l,m}^j\mid\mathcal H])\right|
>\sqrt{2MhV_{j,l}x}+B_jx/3\,\middle|\,\mathcal H\right)\le2e^{-x}.
\end{equation*}
A union bound over the $2\ell$ pairs $(j,l)$ gives failure probability $4\ell e^{-x_\delta}=\delta$. Division by $Mh$ proves Eq.~\eqref{eq:levy_first_bernstein}. Adding and subtracting $m_{1,l}$ and $m_{2,l}-J_\varepsilon$ proves Eq.~\eqref{eq:levy_diffusion_propagation}.
\end{proof}

\paragraph{Small-time limits and tail correction.}
Let $Y=\Delta v_h\mathbf1_{\{|\Delta v_h|<\varepsilon\}}$. The local generator yields its moment limits for every $0<\alpha<2$. At a fixed augmented state $x=(v,\bm z)$, set $b=F(v)-\bm u^\top\bm z$. For a bounded smooth function $f$ of the velocity increment, with $f(0)=0$,
\begin{equation}
\mathcal Af(0)=bf'(0)+\frac{\sigma^2(v)}2 f''(0)
+\int_{\mathbb R}\left[f(r)-r f'(0)\mathbf1_{\{|r|<1\}}\right]\nu_{\alpha,\kappa}(\mathrm{d}r).
\label{eq:local_generator_proof}
\end{equation}
To justify the smooth-test limit, stop when $|v_{t+s}-v|\ge R$. The continuous coefficients and the memory convolution are bounded before stopping on a sufficiently short interval; the latter is continuous by $L^1$ translation continuity of $\gamma$. Localized It\^o's formula and dominated convergence give the generator at the release history. A smooth exit cutoff, constant near zero and outside radius $R$, gives
\begin{equation*}
\limsup_{h\downarrow0}h^{-1}\mathbb P(\tau_R\le h\mid\mathcal H_t^v)
\le\int_{\mathbb R}\min\{1,4r^2/R^2\}\nu_{\alpha,\kappa}(\mathrm{d}r)
\longrightarrow0\quad(R\to\infty).
\end{equation*}
Its first two derivatives vanish at zero. Removing stopping therefore preserves the smooth-test limit. For $f_j(r)=r^j\mathbf1_{\{|r|<\varepsilon\}}$, $j=1,2,4$, smooth only near $\pm\varepsilon$. Smooth majorants of this shell bound the approximation error after division by $h$ by a constant times its jump measure. Taking $h\downarrow0$ and then the shell width to zero is valid because $\nu_{\alpha,\kappa}$ has no atoms. Symmetry and direct integration therefore yield
\begin{align*}
\lim_{h\downarrow0}h^{-1}\mathbb EY&=b,&
\lim_{h\downarrow0}h^{-1}\mathbb EY^2&=\sigma^2(v)+J_\varepsilon,\\
\lim_{h\downarrow0}h^{-1}\mathbb EY^4
&=\int_{|r|<\varepsilon}r^4\nu_{\alpha,\kappa}(\mathrm{d}r)
=\frac{2c_\alpha\kappa^\alpha}{4-\alpha}\varepsilon^{4-\alpha}.
\end{align*}
These are conditional limits at the fixed state, as in the nonlocal Kramers--Moyal identities~\cite{li2021levy}. For $j=1$, the compensated integrand vanishes near zero and is odd on the remaining symmetric annuli, so its integral is zero for every $0<\alpha<2$. Since $\mathbb EY=O(h)$ and $\mathbb EY^2=O(h)$, their squared expectations divided by $h$ vanish, proving Eq.~\eqref{eq:levy_variance_limits}. The localized argument also applies directly to the original convolution conditioned on its complete history, without an exact finite-dimensional embedding.

For the pooled count, suppose all replicas are conditionally independent and set
$\Lambda_{\varepsilon,h}=(\ell h)^{-1}\sum_l\mathbb P(|\Delta v_{l,1}|\ge\varepsilon\mid\mathcal H)$.
The sum of the Bernoulli variances is at most $\ell Mh\Lambda_{\varepsilon,h}$. Hence, for $x>0$, with failure probability at most $2e^{-x}$,
\begin{equation}
|\widehat\Lambda_\varepsilon-\Lambda_\varepsilon|
\le\sqrt{\frac{2\Lambda_{\varepsilon,h}x}{\ell Mh}}+\frac{x}{3\ell Mh}
+|\Lambda_{\varepsilon,h}-\Lambda_\varepsilon|.
\label{eq:tail_count_error}
\end{equation}
Eq.~\eqref{eq:stable_correction_sensitivity} follows by substituting
$\widehat J_\varepsilon=\varepsilon^2\widehat\alpha\widehat\Lambda_\varepsilon/(2-\widehat\alpha)$ and subtracting the analogous identity for $J_\varepsilon$. It separates count error, finite-lag tail bias, and the index-estimation error.

\paragraph{Amplitude and forward errors.}
For any comparator $\theta$, the optimization residual and two triangle inequalities give
\begin{align*}
\|f_{\widehat\theta}(\bm v)-f_\star(\bm v)\|_\ell
&\le\|f_{\widehat\theta}(\bm v)-\widehat{\bm D}^{(2)}\|_\ell
+\|\widehat{\bm D}^{(2)}-f_\star(\bm v)\|_\ell\\
&\le\|f_\theta(\bm v)-f_\star(\bm v)\|_\ell
+2\|\widehat{\bm D}^{(2)}-f_\star(\bm v)\|_\ell+\eta_{\rm opt}.
\end{align*}
Taking the infimum proves Eq.~\eqref{eq:diffusion_fit_decomposition}. The pointwise inequality $(a-b)^2\le|a^2-b^2|$ for $a,b\ge0$ and Cauchy--Schwarz prove the amplitude bound.

Write $I_l(\gamma)=\int_0^{t_l}\gamma(t_l-s)v_s\,\mathrm{d}s$. The exact decomposition is
\begin{equation}
\widehat{\bm y}=\bm\Phi\bm c_N+\bm d+\bm e,\qquad
e_l=m_{1,l}-\widehat D_l^{(1)},\qquad
d_l=b_l-m_{1,l}+I_l(\gamma)-(\bm\Phi\bm c_N)_l.
\label{eq:inverse_data_decomposition}
\end{equation}
Cauchy--Schwarz gives $|I_l(\gamma)-I_l(\gamma_N)|\le\|v\|_{L^2}\eta_N$, and therefore $\|\bm d\|_2\le\sqrt\ell\,r_{N,h}$. To make the quadrature term explicit, let $v_h(s)=v_j$ on $[t_j,t_{j+1})$ and let $\omega_{\gamma_N}(h)$ be the modulus of continuity of $\gamma_N$. Adding and subtracting the grid value of the kernel in each integral gives
\begin{equation}
q_{N,h}\le\omega_{\gamma_N}(h)\|v\|_{L^1}
+\|\gamma_N\|_\infty\|v-v_h\|_{L^1}\longrightarrow0
\quad(h\downarrow0,\ N\text{ fixed}).
\label{eq:cadlag_quadrature}
\end{equation}
A c\`adl\`ag path on $[0,T]$ is bounded and $v_h\to v$ at its continuity points, hence in $L^1$ by dominated convergence.  For spline kernels, $\omega_{\gamma_N}(h)\le h\|\gamma_N'\|_\infty$.

For fixed prepared histories with the same release velocity $v_l$, let $\widetilde v^{(l)}$ be the prepared history, $\widetilde b_l$ its instantaneous drift, and $\widetilde m_{1,l}$ its exact finite-lag statistic. Then
\begin{align}
|b_l-\widetilde b_l|
&\le\|\gamma\|_{L^2(0,t_l)}\|v-\widetilde v^{(l)}\|_{L^2(0,t_l)},\nonumber\\
\beta_{1,h}
&\le\max_l|b_l-\widetilde b_l|
+\max_l|\widetilde b_l-\widetilde m_{1,l}|
+\max_l|\widetilde m_{1,l}-m_{1,l}|.
\label{eq:history_bias_propagation}
\end{align}
The terms represent history mismatch, finite lag under the prepared history, and numerical or observation effects on the actual statistic. If increments are measured from their own perturbed release velocity, the drift mismatch additionally includes $|F(v_l)-F(\widetilde v_l)|$. If they are instead centered at $v_l$, the release offset must also be included in the finite-lag statistic; it contributes to the increment statistic as well as to the drift mismatch. In the jump case, numerical perturbations of the truncated statistic include threshold crossings.

\begin{proof}[Proof of Theorem~\ref{thm:resolved_recovery}]
We combine a bound for the sampling error in the observed subspace with the exact spectral solution. The independent centered row errors in Eq.~\eqref{eq:inverse_data_decomposition} satisfy
\begin{equation*}
\mathbb E[\bm e\mid\mathcal H]=0,\qquad
\operatorname{Cov}(\bm e\mid\mathcal H)=\frac{\operatorname{diag}(V_{1,l})}{Mh}.
\end{equation*}
Writing $\bm w=\bm U_r^\top\bm e$, we have
\begin{equation*}
\mathbb E[\|\bm w\|_2^2\mid\mathcal H]
=\frac{\operatorname{tr}(\bm U_r^\top\operatorname{diag}(V_{1,l})\bm U_r)}{Mh}
\le\frac{r\overline V_1}{Mh}.
\end{equation*}
Markov's inequality yields the Brownian choice in Eq.~\eqref{eq:projected_sampling_bound}. For frozen Gaussian increments, Gaussian concentration instead permits
$\mathfrak z_\delta=\overline\sigma(\sqrt r+\sqrt{2\log(1/\delta)})/\sqrt{Mh}$,
where $\overline\sigma=\max_l|\sigma_l|$.
For the L\'evy branch, fix a unit vector $\bm a\in\mathbb R^r$ and put $\bm q=\bm U_r\bm a$. Since $\|\bm q\|_2=1$, the independent summands of $\bm a^\top\bm w$ are bounded by $2\varepsilon/(Mh)$ and their total variance is at most $\overline V_1/(Mh)$. Thus
\begin{equation*}
\mathbb P\!\left(|\bm a^\top\bm w|>
\sqrt{\frac{2\overline V_1x}{Mh}}+\frac{2\varepsilon x}{3Mh}\,\middle|\,\mathcal H\right)\le2e^{-x}.
\end{equation*}
A $1/2$-net of the unit sphere in $\mathbb R^r$ has at most $5^r$ points and bounds the norm by twice the largest directional value. The union bound has failure probability $2\cdot5^re^{-x_{r,\delta}}=\delta$, giving the L\'evy choice of $\mathfrak z_\delta$. Either branch therefore provides a single event $\|\bm w\|_2\le\mathfrak z_\delta$ with conditional probability at least $1-\delta$.

The minimum-norm solution and its error are
\begin{align}
\widehat{\bm c}_\lambda
&=\bm V_r\operatorname{diag}\!\left(\frac{s_i^3}{s_i^4+\lambda}\right)\bm U_r^\top\widehat{\bm y},\nonumber\\
\widehat{\bm c}_\lambda-\bm c_N
&=\bm V_r\operatorname{diag}\!\left(\frac{s_i^3}{s_i^4+\lambda}\right)
(\bm w+\bm U_r^\top\bm d)-\bm R_\lambda\bm c_N.
\label{eq:coefficient_error_identity}
\end{align}
Indeed, the data and penalty both vanish on $\ker\bm\Phi$, and the minimum-norm choice sets that component to zero. Multiplying the error identity by $\bm G_N^{1/2}$ and adding the projection error gives, on the sampling event,
\begin{align*}
\|\widehat\gamma_\lambda-\gamma\|_{L^2}
&\le\eta_N+\sqrt{g_N^+}\rho_{N,\lambda}
(\|\bm w\|_2+\|\bm d\|_2)+\mathcal B_{N,\lambda}\\
&\le\eta_N+\sqrt{g_N^+}\rho_{N,\lambda}
(\mathfrak z_\delta+\sqrt\ell\,r_{N,h})+\mathcal B_{N,\lambda}.
\end{align*}
For full column rank, $\|\bm R_\lambda\|_2=\lambda/(\min_i s_i^4+\lambda)$ and
$\|\bm c_N\|_2\le(g_N^-)^{-1/2}\|\gamma\|_{L^2}$, proving the stated simplification of the bias. The sampling event is common to all $\lambda>0$ and hence covers a data-selected $\widehat\lambda$.
\end{proof}

\end{document}